\documentclass[12pt]{article}
\usepackage{amsmath,amssymb,amsfonts,amsthm}
\usepackage{eucal}

\newcommand{\C}{\mathbb{C}}

\newcommand{\Z}{\mathbb{Z}}
\newcommand{\cO}{\mathcal{O}}

\newcommand{\Pt}{\mathbb{P}^2}

\newcommand{\zee}{\mathfrak{z}}

\newcommand{\Gp}{\Gamma_+}
\newcommand{\Gm}{\Gamma_-}

\newcommand{\Om}{\Omega}

\newcommand{\cU}{\mathcal{U}}

\DeclareMathOperator{\End}{End}
\DeclareMathOperator{\Hilb}{Hilb}

\DeclareMathOperator{\ch}{ch}

\newtheorem*{thma}{Theorem A}
\newtheorem{thm}{Theorem}

\newtheorem{lemma}{Lemma}
\newtheorem{cor}{Corollary}

\author{Erik Carlsson, Simons center for geometry and physics}

\title{Hall-Littlewood polynomials and vector bundles on the Hilbert scheme}

\begin{document}

\maketitle

\abstract{
Let $E$ be the bundle defined by applying a polynomial representation
of $GL_n$ to the tautological bundle on the Hilbert
scheme of $n$ points in the complex plane.
By a result of Haiman \cite{H3}, the
Cech cohomology groups $H^i(E)$ vanish for all
$i>0$. It follows that the equivariant Euler characteristic
with respect to the standard two-dimensional torus action has
nonnegative coefficients in the torus variables $z_1,z_2$,
because they count the dimensions of the weight spaces of $H^0(E)$.
We derive a very explicit asymmetric formula for 
this Euler characteristic which has this property, 
by expanding known contour integral formulas 
for the Euler characteristic stemming from the quiver description \cite{Ne,N}
in $z_2$, and calculating the coefficients using
Jing's Hall-Littlewood vertex operator
with parameter $z_1$ \cite{J}. 
%These results give new formulas for the
%character of component of the 
%trivial representation in the polygraph ring \cite{H3}, 
%which are not obviously connected to Haiman's formulas
%expressed in terms of Macdonald polynomials.
}

%\abstract{We derive a formula for the equivariant Euler characteristic
%%of the bundle built by applying a polynomial representation to the 
%tautological bundle on the Hilbert scheme of $n$ points on $\C^2$,
%endowed with a standard action of a two-dimensional complex torus
%$T=\{z_1,z_2\}$.
%Unlike localization formulas, this expression describes the signed dimension
%of the weight spaces as a nonnegative integer, which reflects
%a result of Haiman that the higher cohomology groups of any such bundle
%vanish \cite{H3}. Our formula is obtained by expanding contour
%integral formulas for the Euler characteristic \cite{Ne,N} in the
%$z_2$ variable, and expressing the coefficients using Hall-Littlewood
%polynomials with parameter $z_1$, and a vertex operator due to Jing \cite{J}.
%}

\section{Introduction}

Let $\Hilb_n \C^2$ denote the Hilbert scheme of $n$ points in the complex plane,
and consider the standard two-dimensional torus action induced from
\begin{equation}
\label{torus}
T = (\C^*)^2 \circlearrowright \C^2,\quad 
(z_1,z_2)\cdot (x,y) = (z_1^{-1}x,z_2^{-1}y)
\end{equation}
by pullback of ideals. We also have an $n$-dimensional tautological bundle $\cU$
on the Hilbert scheme, whose fiber over a subscheme $[Z] \in \Hilb_n \C^2$
is simply the space of sections of $\cO_Z$, and which inherits an action
of $T$. See \cite{Nak2} for details.

Given a representation $\rho$ of $GL_n$, we obtain a new
equivariant bundle $E = \rho(\cU)$, and we may consider
its Cech cohomology groups $H^i(E)$, as well as its equivariant Euler characteristic
\[\chi_n(E) = \sum_i (-1)^i \ch H_{\Hilb_n}^i (E) \in \C[z_1^{\pm 1}, z_2^{\pm 1}],\]
where $\ch$ denotes the character of $H^i(E)$ as a representation of $T$.
If $\Lambda$ is the ring of symmetric polynomials in infinitely many variables,
the polynomial representations are in the image of the map
\[ \Lambda \rightarrow K_T(\Hilb_n \C^2),\quad 
s_\mu \mapsto \mathbb{S}_\mu(\cU),\]
where $s_\mu \in \Lambda$ is the Schur polynomial, $\mathbb{S}_\mu$ is
the corresponding representation of $GL_n$ (the Schur functor), and $\mu$ 
is a partition. Since the Euler characteristic is defined at the level of $K$-theory,
we have a well defined Euler characteristic $\chi_n(f(\cU))$, for any symmetric
function $f \in \Lambda$.

The main result of this paper is the following formula for the Euler
characteristic,
\begin{thma}
The Euler characteristic is given by
\[\chi_n(f(\cU)) =  \sum_{\mu,\nu} z_2^{|\mu|}z_1^{|\mu|+k_{\mu\nu}}
b_{\nu,n}(z_1)^{-1} f_{\nu\mu}(z_1).\]
\label{introthm}
\end{thma}
\noindent
Here $k_{\mu\nu}$ is an integer, $b_{\nu,n}(z)$ is the norm squred
of the Hall-Littlewood polynomial $P_\nu(X;z)$ in $n$ variables,
and $f_{\nu\mu}(z)$ is the matrix element of the operator of multiplication
by $f$ in the Hall-Littlewood basis.
The significance of this formula
is that the coefficients of its power series 
about the origin are
\emph{nonnegative} integers whenever $f$ is an honest representation,
i.e. a nonnegative integral linear combination
of Schur polynomials. This nonnegativity follows from a result
of Haiman which says that
the Cech cohomology groups $H^i(\cU^{\otimes l}\otimes P)$ vanish for $i>0$, 
where $P$ is the Procesi bundle \cite{H3}. Since the trivial bundle
is a summand of the Procesi bundle, and every $E=f(\cU)$ appears
as a summand of $\cU^{\otimes l}$, it follows that the desired
Euler characteristic is the character of the honest representation $H^0(E)$.

It is not clear how this formula relates to the special case of Haiman's
formulas, corresponding to the trivial component of the Procesi bundle.
Haiman's answers are expressed
in terms of the Macdonald polynomials,
via the isomorphism of Bridgeland, King, and Reid \cite{BKR}.
%(theorem \ref{introthm} calculates the component of the trivial
%representation of $S_n$ of this module).
Both sets of formulas are expressed in terms of symmetric functions,
but in our formula, the rank of the Hilbert scheme corresponds to the 
number of variables, whereas in Haiman's formulas, the number of variables is
infinite, and $n$ corresponds to the degree.
The combinatorics of Macdonald polynomials are of course more difficult,
but we can check the agreement
when the bundle is trivial, explained in corollary \ref{thecor}.

It is interesting to note that the Euler characteristic is symmetric
in $z_1,z_2$, which is not obvious from theorem A.
This is also the case in the well-known $q,t$-Catalan number
formulas studied by Garsia, Haglund, and Haiman \cite{GHl,GH},
which the author first learned about from Gorsky, Mazin, and Shende.
In fact, Hall-Littlewood polynomials have been used to study
this topic in a paper by Garsia, Xin, and Zabrocki \cite{GXZ}.
It would also be of interest to relate them to
the results and conjectures of Gorsky, Oblomkov, Rasmussen and Shende
\cite{GORS,OS}, and any connections with the Hall-Littlewood formulas
of Mironov, Morozov, Shakirov, and Sleptsov \cite{MMSS}.

Our proof is based on a contour integral formula for the Euler
characteristic \eqref{contour} coming from the quiver description
on the Hilbert scheme, which can be found in 
Negut \cite{Ne}, and is
a $K$-theoretic version of a cohomological formula
by Nekrasov \cite{N}. 
We expand this formula in the $z_2$ variable, 
and calculate the coefficients in terms of the Hall-Littlewood inner product
in $n$ variables, with parameter $z_1$. An essential role is played by a vertex operator due to Jing \cite{J},
which extends Bernstein's vertex operator \cite{Z} from Schur to Hall-Littlewood polynomials.

\emph{Acknowledgments.} The author would like to thank 
the Simons foundation for its support,
as well as Eugene Gorsky, Vivek Shende, Mikhail Mazin, Alexei Oblomkov,
and Andrei Okounkov for many valuable discussions on this topic.

\section{Contour integrals}

The Atiyah-Bott-Lefschetz localization formula gives an explicit
formula for the Euler characteristic defined in the introduction,
\begin{equation}
\label{loc}
\chi_n(f(\cU)) = \sum_{|\mu|=n} \cU_\mu \ch \Om(T^*_\mu) \in \C(z_1,z_2).
\end{equation}
Here the fixed points of $\Hilb_n \C^2$ are indexed by partitions $\mu$ of
$n$, $\cU_\mu, T^*_\mu$ denote the fibers of $\cU$, and the cotangent bundle
respectively, and 
\[\Om(V) = \left( \ch \sum_i (-1)^i \Lambda^i V\right)^{-1} \in \C(z_1,z_2),\]
for a torus representation $V$. More generally,
$\Om$ may be extended to the whole representation ring $\Z(T)$ by
\begin{equation}
\label{Om}
\Om(A+B) = \Om(A)\Om(B),\quad \Om(x) = 
(1-x)^{-1},
\end{equation}
for any monomial $x$. See \cite{H2,Nak2} for combinatorial formulas for
the summands.

Strictly speaking, the localization formula does not apply in this situation
because the Hilbert scheme is not compact. However,
by a result of Nakajima \cite{Nak4},
the weight spaces of the Cech cohomology group are finite-dimensional,
and the Euler characteristic lives in $\Z((z_1,z_2))$, which
represents the expansion of \eqref{loc} about the origin.
If $\rho$ is a polynomial representation, as it is
in this paper, the Euler characteristic
turns out to be holomorphic at the origin. As usual, we cannot
extract the signed dimensions of the weight spaces 
from the localization formula without simplifying
the expression. In fact most of the terms have a singularity along a
one-dimensional curve through the origin in the $z_1,z_2$ plane,
and their expansions change depending on which $z_a$ we expand 
about first.

This issue can be resolved using the following contour integral formula,
\[\chi_n(f(\cU)) = \frac{1}{n!}\Om(1-M)^n 
\oint_{|x_1|=r} \frac{dx_1}{x_1} \cdots \oint_{|x_n|=r} \frac{dx_n}{x_n}\]

\begin{equation}
\label{contour} 
f(X)\Om(\overline{X}) \Om(z_1z_2 X)
\Om(-M \Delta),
\end{equation}
where
\[M = (1-z_1)(1-z_2),\quad X = x_1+\cdots+x_n,\quad\] 
\[\Delta = \sum_{i \neq j} x_i x_j^{-1} = X \overline{X} - n, \quad
f(X) = f(x_1,...,x_n),\quad \overline{x_i} = x_i^{-1}.\]
We refer to \cite{Ne} for an explanation of this formula,
or \cite{N} for the original cohomological version.
These formulas come from the description of the 
Hilbert scheme as a quiver variety, and apply to the more
general moduli space of higher rank sheaves on $\Pt$, see \cite{Nak2}.
They are shown to agree with \eqref{loc} directly by applying
the Cauchy residue formula, one variable at a time.
See also \cite{C}, which produces similar formulas, by considering
the Hilbert scheme as a subvariety of an infinite-dimensional Grassmannian.

Under formula \eqref{contour}, we find that $\chi_n(f(\cU))$
is manifestly holomorphic
at the origin, simply because the expansion of the integrand 
in $z_a$ is valid in the interior of the contour. 
Furthermore, we may
count the signed dimension of the weight spaces by applying the contour integral
to each coefficient. Each such integral may be expressed in terms of the 
standard Hall inner product on Symmetric functions in $n$ variables, establishing
that it is an integer.

\section{Hall-Littlewood polynomials}

Let us recall briefly some notation about Hall-Littlewood polynomials
and the plethystic notation,
which we standardize with chapter 3 of Macdonald's book \cite{Mac}, and
Haiman \cite{H1}.

Let $\Lambda$ denote the ring of symmetric functions, and consider the Hall-Littlewood inner product in finitely many variables,
\begin{equation}
\label{hln}
(f,g)_{z,n} = \frac{1}{n!}[X]_1 f(X)g(\overline{X}) 
\Om(-(1-z)\Delta_n),
\end{equation}
where
\[X = x_1+\cdots +x_n,\quad\Delta_n = \sum_{i \neq j} x_i x_j^{-1},\]
as in the introduction. The constant term $[X]_1$
may be defined either as a contour integral for any fixed value of $z$, 
or by expanding the integrand in $z$, and simply extracting the constant term of each coefficient, which is a Laurent polynomial in $x_i$.
We also have its limit as the number of variables tends to infinity,
normalized so that the norm of $1 \in \Lambda$ is one, defined by
\begin{equation}
\label{hl}
(p_\mu,p_\nu)_z = \delta_{\mu\nu} \zee(\mu) \prod_i (1-z^{\mu_i})^{-1},
\end{equation}
where $p_\mu$ are the symmetric power sums
\[p_\mu = \prod_k p_{\mu_k},\quad p_k = x_1^k+x_2^k+\cdots.\]

The Hall-Littlewood polynomials $P_\mu(X;z)$ for partitions of length
$\ell(\mu) \leq n$ constitute an orthogonal basis for \eqref{hln},
and satisfy
\begin{equation}
(P_\mu,P_\nu)_{z,n} = \delta_{\mu\nu} (1-z)^n b_{\mu,n}(z)^{-1},
\end{equation}
where
\[b_{\mu,n}(z) = \prod_{i \geq 0} [m_i(\mu)]_z,\quad   
[k]_z = \prod_{1 \leq j \leq k}(1-z^j),\]
and $m_i(\mu)$ is the number of times that $i$ appears in $\mu$, with the
multiplicity of zero defined as $n-\ell(\mu)$. In the limit as $n$ tends to
infinity, we get
\[(P_\mu,P_\nu)_z = \delta_{\mu\nu} b_\mu(z)^{-1},\]
\[b_\mu(z) = \lim_{n \rightarrow \infty} b_{\mu,n}(z)b_{\emptyset,n}(z)^{-1} =
\prod_{i \geq 1} [m_i(\mu)]_z.\]

Given a rational function $A$ in some set of variables, let $A_k$ denote the evaluation
at $z=z^k$, for each indeterminant $z$ that appears in $A$. 
We will make heavy use of the following multiplication operator
\[\Gm(A) : \Lambda \rightarrow \Lambda,\quad \Gm(A) g = \exp \left(\sum_{k \geq 1}\frac{1}{k} f_k p_k\right) g,\]
which is technically only defined as a power series in 
all variables present in $A$ with values in $\End(\Lambda)$.
This is a harmless issue for our purposes, but
see Frenkel and Ben-Zvi \cite{FB} for a full exposition.
For instance, we have
\[\Gm(x) \cdot 1 = \sum_{k \geq 0} x^k h_k.\]
Its dual under the standard Hall inner product is a ring homomorphism, defined on generators by
\[\Gp(A) p_k = p_k+A_k.\]

The following relations are easily verified,
\[(\Gm(A)f,g)_z = (f,\Gp(A(1-z)^{-1}))_z,\]
\[\Gamma_{\pm}(A+B) = \Gamma_{\pm}(A) \Gamma_{\pm}(B),\quad
\Gamma_{\pm}^m(A) = \Gamma_{\pm}(mA),\]
\[\Gamma_{\pm}(A) x^d = x^d\Gamma_{\pm}(x^{\pm 1} A),\quad
x^d\cdot p_\mu = x^{|\mu|} p_\mu,\]
\begin{equation}
\label{CR}
\Gp(A)\Gm(B) = \Gm(B) \Gp(A) \Om(AB),
\end{equation}
where
\[\Om(A) = \exp\left( \sum_{k \geq 1} \frac{1}{k} A_k\right),\]
and the convergence of $\Om(AB)$ puts restrictions on $A,B$.
Notice that this definition of $\Om(A)$ is consistent with \eqref{Om}.

Next, we recall Jing's vertex operator \cite{J}, 
which generates the Hall-Littlewood polynomials by successive 
applications to $1\in \Lambda$.
In this notation, it is defined by
\begin{equation}
J_k = [x^k] \Gm(x(1-z))\Gp^{-1}(x^{-1}),
\label{jing}
\end{equation}
and has the property that
\[Q_\mu = J_{\mu_1} \cdots J_{\mu_n} \cdot 1,\]
where
\[Q_\mu(X;z) = b_\mu(z) P_\mu(X;z),\]
is the dual basis to $P_\mu$ under \eqref{hl},
as in MacDonald's book. Upon setting $z=0$,
it becomes the vertex operator defined
by Bernstein \cite{Z}, which acts on the Schur polynomials.

Given any operator $\varphi : \Lambda \rightarrow \Lambda$, 
let us label its matrix elements in the Hall-Littlewood basis by
\[ \varphi \cdot P_{\nu}(X;z) = \sum_\mu \varphi_{\mu\nu}(z) P_{\mu}(X;z).\]
If $f \in \Lambda$ is a polynomial, then we define $f_{\mu\nu}(z)$ 
to be the matrix elements of multiplication by $f$.
We will also set
\[ \psi_{\mu\nu}(z) = \Gm(1)_{\mu\nu}(z),\]
which is the same thing as multiplication by the complete 
symmetric polynomial $h_k$, for $k=|\nu|-|\mu|$.
The Pieri rules for Hall-Littlewood polynomials provide a combinatorial 
description of these coefficients, which we will not need.

\section{Proof of the theorem}

We may now state and prove our main result:

\begin{thm}
\label{thethm}
If $f \in \Lambda$ is a symmetric function, and $\cU$ 
is the tautological $n$-dimensional
bundle on $\Hilb_n \C^2$ with the torus action \eqref{torus}, then we have
\begin{equation}
\label{maineq}
\chi_n(f(\cU)) =  \sum_{\mu,\nu} z_2^{|\mu|}z_1^{|\mu|+k_{\mu\nu}}
b_{\nu,n}(z_1)^{-1} f_{\nu\mu}(z_1),
\end{equation}
where
\[k_{\mu\nu} = \sum_i \left(\begin{array}{c} \mu'_i \\ 2 \end{array}\right)+
\left(\begin{array}{c} \nu'_i \\ 2 \end{array}\right)-\mu'_i \nu_i',\]
and $\mu'$ is the conjugate partition to $\mu$.
\end{thm}

If $f$ is a linear combination of the Schur polynomials 
with nonnegative integer coefficients,
then $f_{\mu\nu}(z)$ is polynomial in $z$ with nonnegative integer coefficients. 
The coefficients of the power series expansion of \eqref{maineq} 
are therefore nonnegative integers, representing the result of Haiman
explained the introduction that the higher Cech cohomology groups vanish.

An immediate corollary is the well-known formula for the 
the space of sections of $\cO$.

\begin{cor}
\label{thecor}
The space of sections of the trivial bundle is given by
\begin{equation}
\label{Z}
\sum_{n\geq 0} q^n \chi_{n} (\cO) = \Om(qM^{-1}) = 
\prod_{i,j \geq 0} (1-z_1^iz_2^j q)^{-1}.
\end{equation}
\end{cor}

\begin{proof}
The partition function corresponds to $f=1$, whence $f_{\mu\nu} = \delta_{\mu\nu}$.
We may easily check that $k_{\mu\mu} = -|\mu|$, so that formula \eqref{maineq} becomes
\[\chi_{n}(\cO)= \sum_{\mu} z_2^{|\mu|} b_{\mu,n}(z_1)^{-1}.\]
Fixing $n$, we may associate to $\mu$ another partition 
$\tilde{\mu}$, whose terms are the multiset of positive
integers $m_i(\mu)$,
including the multiplicity of zero, as defined above.
We can rewrite the above expression as
\[\sum_{|\lambda|=n} \sum_{\tilde{\mu} =\lambda} z_2^{|\mu|} \prod_{k} 
[\lambda_k]_{z_1}^{-1}.\]
One may easily check that
\[\sum_{\tilde{\mu} = \lambda} z_2^{|\mu|} = m_{\lambda}(1,z_2,z_2^2,...),\quad [\lambda_k]_{z_1}^{-1} = h_{\lambda_k}(1,z_1,z_1^2,...),\]
where $m_\mu,h_\mu$ are the monomial and complete symmetric polynomials. Then
\[ \chi_n(\cO) = \sum_{|\lambda|=n} m_\lambda(1,z_2,z_2^2,...) h_{\lambda} (1,z_1,z_1^2,...).\]
Converting this to \eqref{Z} is precisely chapter I, formula (4.2) of \cite{Mac}.
\end{proof}

Before proving the theorem, we need a technical lemma:

\begin{lemma}
\label{thelemma}
We have
\begin{equation}
\label{lemmaeq}
\sum_{\lambda} z^{-|\lambda|} b_\lambda(z) 
\psi_{\mu\lambda}(z)\psi_{\nu\lambda}(z) = z^{k_{\mu\nu}}.
\end{equation}
\end{lemma}
\begin{proof}

The exponent $k_{\mu\nu}$ satisfies
\[k_{\emptyset \emptyset} = 0,\quad k_{\mu\nu} = k_{\nu\mu},\quad
k_{[a,\mu] \nu} - k_{\mu\nu} = |\mu|-|\nu|,\]
\begin{equation}
\label{kprop}
a \geq \mu_1,\nu_1, \quad [a,\mu] = [a,\mu_1,...,\mu_n].
\end{equation}
It is uniquely determined by these properties by successively adding terms to $\mu,\nu$, in increasing order.

Now, we may rewrite the left hand side of \eqref{lemmaeq} as
\[\sum_{\lambda} z^{-|\lambda|} (\Gm(1) P_\lambda, Q_\mu)_{z} 
(\Gm(1) P_\lambda,Q_\nu)_{z} (P_\lambda,P_\lambda)_z^{-1} = (Q_\mu,Q_\nu)',\]
where
\[(f,g)' = \left(\Gp(A^{-1}) f, z^{-d} \Gp(A^{-1}) g\right)_z, \quad A =1-z,\]
and $d$ is the operator of multiplication by the norm on $\Lambda$.
It suffices to prove that this inner product satisfies
\[ (J_a f, Q_\nu)' = (z^d f,z^{-d} Q_\nu)' \]
whenever $a\geq \nu_1$, which would establish the last property in \eqref{kprop}.

Inserting the definition \eqref{jing}, we get the coefficient of $x^a$ in
\[\left(\Gp(A^{-1}) \Gm(xA)\Gp^{-1}(x^{-1})f, z^{-d} \Gp(A^{-1}) Q_\nu\right)_z =\]
\[(1-x)^{-1}\left(\Gp(A^{-1}-x^{-1})f, \Gp(x)z^{-d} \Gp(A^{-1}) Q_\nu\right)_z =\]
\[(1-x)^{-1}\left(\Gp(A^{-1}-x^{-1})f, z^{-d} \Gp(A^{-1}+xz^{-1}) Q_\nu\right)_z.\]
%
%\[\left(\Gm(z)\Gm(x(1-z)) \Gp^{-1}(x^{-1})f, z^{-d} \Gm(z) Q_\nu\right)_z =\]
%
%\[\left(\Gm(z)\Gp^{-1}(x^{-1})f, \Gp(x)z^{-d} \Gm(z) Q_\nu\right)_z =\]
%
%\[(1-x)^{-1} \left(\Gm(z)\Gp^{-1}(x^{-1})f, z^{-d} \Gm(z) \Gp(xz^{-1}) Q_\nu\right)_z.\]
%
%\[\left(\Gp((1-z)^{-1}) \Gm(x(1-z)) \Gp^{-1}(x^{-1})f, z^{-d} \Gp((1-z)^{-1}) m_\nu\right) =\]
%
%\[ (1-x)^{-1} \left(\Gp\left((1-z)^{-1}-x^{-1}\right)f, z^{-d} \Gp\left((1-z)^{-1}+z^{-1}x\right) m_\nu\right) =\]
%
%\[(1-x)^{-1}g(x,z),\]
%
using the vertex operator relations \eqref{CR}.

The final expression may be written as $(1-x)^{-1} F(x)$, where $F(x)$ is a Laurent polynomial in $x$ with 
coefficients in $\C(z)$.
%In fact, since
%
%\[(f,\Gp(x) m_\nu) = (\Gm(x) f, m_\nu) = \sum_{i \geq 0} x^i (fh_i,m_\nu)\]
%
%and $m_\mu$ is dual to $h_\mu$, we see that 
%
Since $Q_\nu$ is lower-triangular with respect to the monomial basis $m_\mu$, we may bound the degree of $F(x)$ in $x$ by
\[\deg_x F(x) \leq \deg_x \Gp(x) Q_{\nu} \leq 
\max_{\mu \leq \nu} \deg_x \Gp(x) m_{\mu},\]
where $\mu \leq \nu$ refers to the dominance ordering.
Since 
\[(f,\Gp(x) m_\mu) = (\Gm(x) f, m_\mu) = \sum_{i \geq 0} x^k(fh_k,m_\mu),\]
and $h_\mu,m_\mu$ are dual bases, we find that the degree of $F(x)$ is at most $\nu_1$.

Then for $a \geq \nu_1 \geq \deg_x F(x)$, we have
\[[x^a] (1-x)^{-1} F(x) = F(1) = \]
\[\left(\Gp\left(zA^{-1}\right)f, z^{-d} \Gp\left(z^{-1}A^{-1}\right) Q_\nu\right) =(z^{d} f, z^{-d} Q_\nu)' .\]

\end{proof}

We can now prove the main result.

\begin{proof}

We may rewrite the contour integral formula from the introduction as
\begin{equation}
\label{contour2} 
\chi_n(f(\cU)) = \frac{1}{n!}\Om(1-M)^n 
[X]_1
f(X)\Om(\overline{X}) \Om(z_1z_2 X)
\Om(-M \Delta),
\end{equation}
where the constant term is taken from each term in the expansion
of the integrand in $z_1,z_2$. 
Let us expand the rightmost term in the $z_2$ variable,
\[\Om(-M\Delta) = \Om(-(1-z_1)\Delta)\Om(z_2(1-z_1)\Delta) =\]
\[\Om(-(1-z_1)\Delta)\Om(z_2(1-z_1))^{-n}
\sum_{\lambda} z_2^{|\lambda|} b_{\lambda}(z_1)P_{\lambda}(X;z_1) P_{\lambda}(\overline{X};z_1),\] 
by the expansion
\[\Om(x(1-z)XY) = \sum_{\lambda} x^{|\lambda|} b_\lambda (z)
P_{\lambda}(X;z) P_\lambda(Y;z),\]
which can be found in \cite{Mac}, chapter III, equation (4.4). Inserting this into equation \eqref{contour2}, and using \eqref{hln},
we get

%\[\Omega(-z_1)^n \frac{1}{n!}[x_i^0] \Omega\left((X\overline{X}-n)(1-z_1)\right)
%\Omega\left(-X \overline{X}z_2(1-z_1)\right)
%f(X)g(\overline{X}) = \]
%
%\[(1-z_1)^n\sum_{\mu} z_2^{|\mu|} (P_{\mu,z_1} f, P_{\mu,z_1} g)_{1-z_1,n}.\]
%
%Inserting this into formula \ref{modthm}, we get
%
\[(1-z_1)^{-n}\sum_{\lambda} z_2^{|\lambda|} b_\lambda(z_1)
(\Gm(z_1z_2) P_{\lambda} f, \Gm(1)P_{\lambda})_{z_1,n} = \]
\[\sum_{\lambda,\mu,\nu}  
z_2^{|\mu|} z_1^{|\mu|-|\lambda|}
b_\lambda(z_1) b_{\nu,n}(z_1)^{-1}
f_{\nu\mu}(z_1)\psi_{\mu\lambda}(z_1) \psi_{\nu\lambda}(z_1).\]
%
%\[\sum_{\mu,\nu} z_2^{|\mu|}z_1^{|\mu|+k_{\mu\nu}}
%b_{\nu,n}(z_1)^{-1} f_{\nu\mu}(z_1) =\]
%
The result now follows by summing over $\lambda$,
and applying lemma \ref{thelemma}.
\end{proof}

\end{document}